\documentclass[11pt,a4paper,usenames,dvipsnames,reqno]{amsart}
\usepackage{amssymb}
\usepackage{amscd}
\usepackage{amsfonts}
\usepackage[margin=1.25in]{geometry}

\usepackage[all,arc]{xy}
\usepackage{enumerate}
\usepackage{mathrsfs}
\usepackage{color}   
\usepackage{hyperref}
\hypersetup{
    colorlinks=true, 
    linktoc=all,     
    linkcolor=blue,  
}
\usepackage{amsthm}
\usepackage{amsmath}
\usepackage{graphicx}
\usepackage{wasysym}
\usepackage{pgf,tikz}
\usetikzlibrary{positioning}
\usepackage{tikz-cd}
\usepackage{bbm}

\newcommand{\cc}{\mathbb C}

\newcommand{\zz}{\mathbb Z}

\newcommand{\rr}{\mathbb R}
\newcommand{\A}{\mathbb A}

\newcommand{\la}{\langle}
\newcommand{\ra}{\rangle}
\newcommand{\lra}{\longrightarrow}
\newcommand{\hra}{\hookrightarrow}

\newcommand{\al}{\alpha}
\newcommand{\sig}{\sigma}
\newcommand{\be}{\beta}
\newcommand{\ga}{\gamma}
\newcommand{\de}{\delta}
\newcommand{\De}{\Delta}
\newcommand{\Ga}{\Gamma}
\newcommand{\ep}{\epsilon}
\newcommand{\lam}{\lambda}
\newcommand{\Lam}{\Lambda}
\newcommand{\vp}{\varpi}
\newcommand{\ka}{\kappa}

\DeclareMathOperator{\Gal}{Gal}

\DeclareMathOperator{\End}{End}

\DeclareMathOperator{\G}{G}
\DeclareMathOperator{\GL}{GL}

\DeclareMathOperator{\Hom}{Hom}

\DeclareMathOperator{\U}{U}

\DeclareMathOperator{\SO}{SO}
\DeclareMathOperator{\RO}{RO}
\DeclareMathOperator{\SRO}{SRO}

\DeclareMathOperator{\Sp}{Sp}

\DeclareMathOperator{\Lie}{Lie}

\DeclareMathOperator{\Ad}{Ad}

\DeclareMathOperator{\Orb}{Orb}
\DeclareMathOperator{\val}{val}
\DeclareMathOperator{\inv}{inv}
\DeclareMathOperator{\supp}{supp}

\newcommand{\fu}{\mathfrak u}
\newcommand{\fg}{\mathfrak g}
\newcommand{\ft}{\mathfrak t}

\newcommand{\fh}{\mathfrak h}

\newcommand{\fgl}{\mathfrak{gl}}

\newcommand{\calc}{\mathfrak{C}}

\newcommand{\calq}{\mathcal{Q}}
\newcommand{\calo}{\mathcal{O}}

\newcommand{\cald}{\mathcal{D}}

\newcommand{\Herm}{\mathcal{H}erm}

\newcommand{\Y}{\mathbf{Y}}
\newcommand{\X}{\mathbf{X}}

\newcommand{\Nm}{\mathrm{Nm}}

\newcommand{\iso}{\xrightarrow{\sim}}

\newcommand{\Fbar}{\overline{F}}
\newcommand{\bfun}{\mathbf{1}}

\makeatletter
\def\Ddots{\mathinner{\mkern1mu\raise\p@
\vbox{\kern7\p@\hbox{.}}\mkern2mu
\raise4\p@\hbox{.}\mkern2mu\raise7\p@\hbox{.}\mkern1mu}}
\makeatother

\newtheorem{Thm}{Theorem}[section]
\newtheorem{Prop}[Thm]{Proposition}
\newtheorem{Lem}[Thm]{Lemma}
\newtheorem{Cor}[Thm]{Corollary}

\newtheorem{Conj}[Thm]{Conjecture}

\newtheorem{Notation}[Thm]{Notation}

\theoremstyle{definition}
\newtheorem{Def}[Thm]{Definition}

\theoremstyle{remark}
\newtheorem{Rem}[Thm]{Remark}

\theoremstyle{definition}

\title{Endoscopy for unitary symmetric spaces}
\author{Spencer Leslie}

\date\today

\address{Department of Mathematics, Duke University, 120 Science Drive, Durham, NC, USA}

\email{lesliew@math.duke.edu}

\thanks{This work was partially supported by an AMS-Simons Travel Award and by NSF grant DMS-1902865}

\begin{document}

\begin{abstract}
Motivated by global applications, we propose a theory of relative endoscopic data and transfer factors for the symmetric pair $(U(2n),U(n)\times U(n))$ over a local field. We then formulate the smooth transfer conjecture and fundamental lemma, establish the existence of smooth transfer for many test functions, and prove the fundamental lemma for the symmetric pair $(U(4),U(2)\times U(2))$.
\end{abstract}

\maketitle
\setcounter{tocdepth}{1}
\tableofcontents
\section{Introduction}
This paper initiates a program with the goal of stabilizing the relative trace formula associated to symmetric pairs of the form $(U(2n),U(n)\times U(n))$, where $U(n)$ denotes a unitary group of rank $n$. This is the first example of a relative theory of endoscopy.

Our present aim is to develop a infinitesimal theory of relative endoscopic data and transfer factors for the symmetric space considered over any local field $F$. In particular, we formulate the smooth transfer conjecture and state the fundamental lemma when $F$ is non-archimedean. We establish the existence of smooth transfer for many test functions, and prove the fundamental lemma for the symmetric pair $(U(4),U(2)\times U(2))$ via an explicit computation. In a subsequent paper \cite{leslieFL}, we establish this fundamental lemma in full generality via a different technique. Among other things, this paper lays the necessary groundwork for this computation.

Let us describe the setup and motivation. Let $E/F$ be a quadratic extension of global fields, $\A_E$ and $\A_F$ the associated rings of adeles. Let $W_1$ and $W_2$ be two $d$ dimensional Hermitian spaces over $E$. The direct sum $W_1\oplus W_2$ is also a Hermitian space and we have the embedding of unitary groups $$U(W_1)\times U(W_2)\hra U(W_1\oplus W_2).$$
Let $\pi$ be an irreducible cuspidal automorphic representation of $U(W_1\oplus W_2)_{\A_F}$. Roughly, $\pi$ is said to  be \emph{distinguished} by the subgroup $U(W_1)\times U(W_2)$ if the \emph{period integral}
\begin{equation}\label{period of interest}
\displaystyle\int_{[U(W_1)\times U(W_2)]}\varphi(h)dh
\end{equation}
is not equal to zero for some vector $\varphi\in\pi$. Here, $[H]=H(F)\backslash H(\A_F)$ for any $F$-group $H$. The study of distinction of automorphic representations with respect to certain subgroups is a large and active area of automorphic representation theory, but this particular case has appeared in the recent literature in several distinct ways.
\subsection{Global Motivation}
Motivated by the study of the arithmetic of specials cycles in unitary Shimura varieties, Wei Zhang outlined in his 2018 IAS lecture \cite{youtube} a comparison of a relative trace formula on $\GL(W_1\oplus W_2)$ with one on $U(W_1\oplus W_2)$. On the linear group, we may consider period integrals over the subgroups $\GL(W_1)\times \GL(W_2)$ and $U(W_1\oplus W_2)$\footnote{In fact, one should take a different form of the unitary group. See \cite{GetzWambach} for the construction.}; on the unitary side, one considers periods of the form (\ref{period of interest}). Indeed, Chao Li and Wei Zhang have recently \cite{li2019kudla} established the arithmetic fundamental lemma associated to this comparison, with applications to the global Kudla-Rapoport conjecture and arithmetic Siegel-Weil formula. This comparison fits into the general framework for a relative theory of quadratic base change proposed in \cite{GetzWambach}, which in turn may be understood as a method of proving cases of the emerging relative Langlands program. 

For this comparison to be effective for global applications, several results are needed. To begin, one needs the fundamental lemma and smooth transfer to establish the preliminary comparison. This is already problematic as the unitary relative trace formula is not \emph{stable}: when we consider the action of $U(W_1)\times U(W_2)$ on the {symmetric variety} $U(W_1\oplus W_2)/U(W_1)\times U(W_2)$, invariant polynomials distinguish only \emph{geometric orbits}. Analogously to the Arthur--Selberg trace formula, one must ``stabilize" the orbital integrals arising in the unitary relative trace formula in order to affect a comparison between the two formulas.

The present paper lays the local foundations for this stabilization procedure.  Before describing our results, we note that one implication of the conjectured comparison was recently proved by Pollack, Wan, and Zydor \cite{pollack2019residue} under certain local assumptions via a different technique.

\subsection{Relative endoscopy} We now let $F$ be a local field and let $E/F$ be a quadratic field extension. As above, if $W_1$ and $W_2$ are $d$-dimensional Hermitian vector spaces, then $W=W_1\oplus W_2$ is a $2d$-dimensional Hermitian space with a distinguished involutive linear map: $\sig(w_1+w_2) = w_1-w_2$ for $w_i\in W_i$. This induces an involution on the unitary group $U(W)$ with the fixed point subgroup $U(W)^\sig=U(W_1)\times U(W_2)$. Letting $\fu(W)$ denote the Lie algebra of $U(W)$, the differential of $\sig$ induces a $\zz/2\zz$-grading
\[
\fu(W)=\fu(W)_0\oplus \fu(W)_1,
\]
where $\fu(W)_i$ is the $(-1)^i$-eigenspace of $\sigma$. Then the fixed-point group $U(W_1)\times U(W_2)$ acts via restriction of the adjoint action on $\fu(W)_1$. The pair $(U(W_1)\times U(W_2),\fu(W)_1)$ is called an \emph{infinitesimal symmetric pair} since $\fu(W)_1$ is the tangent space to the symmetric space $${\calq_d}=U(W)/(U(W_1)\times U(W_2))$$
at the distinguished $U(W_1)\times U(W_2)$- fixed point. 

In this paper, we develop a theory of elliptic endoscopy for the pair $(U(W_1)\times U(W_2),\fu(W)_1)$, postponing the theory for the symmetric space ${\calq_d}$ to a later paper. We do this both as the infinitesimal theory is simpler to state and also because the stabilization of the relative trace formula is expected to ultimately reduce to studying this case. This is partially motivated by the analogous reduction by Waldspurger \cite{WaldsLiealgebra} and \cite{Waldspurgerimplies} in the case of the Arthur-Selberg trace formula.

Recall (from \cite[Chapter 3]{RogawskiBook}, for example) that for the unitary group $U(W)$, an elliptic endoscopic datum is a triple $(U(V_a)\times U(V_b),s,\eta)$ where $V_n$ denotes a fixed $n$-dimensional split Hermitian space\footnote{We say Hermitian space $V$ is \emph{split} if there exists an isotropic subspace of maximal possible dimension and a self-dual lattice $\Lam\subset V$. This implies that the associated unitary group is quasi-split.}, 
\[
s\in \hat{U}(W)=\GL_{2d}(\cc)
\]
is a semi-simple element, and 
\[
\eta: \GL_a(\cc)\times \GL_b(\cc)\to \GL_{2d}(\cc)
\]induces an isomorphism between $\GL_a(\cc)\times \GL_b(\cc)$ and the connected component of the centralizer of $s$. Here $a+b=2d$. Motivated by the work of Sakellaridis-Venkatesh \cite{SakVenkperiods} and Knop-Schalke \cite{KnopSchalke} on the dual groups of spherical varieties, we consider the dual group of the symmetric variety ${\calq_d}$, 
\[
\hat{\calq}_d=\Sp_{2d}(\cc)\xrightarrow{\varphi_{d}}\GL_{2d}(\cc)=\hat{U}(W).
\]
While a full theory of relative endoscopy is not yet clear, we may na\"{i}vely expect that elliptic endoscopic spherical varieties of ${\calq_d}$ ought to correspond to spherical varieties of endoscopic groups of $U(W)$ that themselves correspond to centralizers semi-simple elements the dual group $\hat{\calq}_d=\Sp_{2d}(\cc)$. In particular, for each elliptic endoscopic subgroup $${\Sp_{2a}(\cc)\times\Sp_{2b}(\cc)\subset\Sp_{2d}(\cc)},$$ we might expect a diagram

\[
\begin{tikzcd}
\Sp_{2d}(\cc)\ar[r]&\GL_{2d}(\cc)\\
\Sp_{2a}(\cc)\times\Sp_{2b}(\cc)\ar[u]\ar[r]&\GL_{2a}(\cc)\times \GL_{2b}(\cc)\ar[u],
\end{tikzcd}
\]
where the bottom horizontal arrow is the dual group embedding associated to the symmetric space $\calq_a\times \calq_b$ of the group $U(V_{2a})\times U(V_{2b})$.

 Happily, this hope is true once we take certain pure inner forms into account. It is not yet clear to us how to place the above heuristic with dual groups on a firm (and generalizable) footing. Our method for establishing the requisite matching of regular semi-simple orbits and effective definition of transfer factors instead reduces to the endoscopic theory of $U(V_d)$ acting on its Lie algebra. The dual group interpretation outlined above only becomes clear \emph{a postiori}.

Recall that $W=W_1\oplus W_2$ is our $2d$-dimensional Hermitian space and let $\xi=(H,s,\eta)$ be an elliptic endoscopic datum for $U(W_1)$, where $H=U(V_a)\times U(V_b)$ with $d=a+b$. Fix representatives $\{\al\}$ and $\{\be\}$ of the isomorphism classes of Hermitian forms on the underlying vector spaces of $V_a$ and $V_b$. Then for each pair $(\al,\be)$, we have the Lie algebras
\begin{equation}\label{endo spaces}
\fu(V_a\oplus V_\al)\text{   and    }\fu(V_b\oplus V_\be),
\end{equation}
 where $V_\al$ simply denotes the Hermitian space $(V_a,\al)$, and similarly for $V_\be$. Each equipped with a natural involution $\sig_\al$ and $\sig_\be$ and the associated symmetric pairs
\[
\left(U(V_a)\times U(V_\al),\fu(V_a\oplus V_\al)_1\right) \text{ and }\left(U(V_b)\times U(V_\be),\fu(V_b\oplus V_\be)_1\right)
\]
 are lower rank analogues of our initial symmetric pair $\left(U(W_1)\times U(W_2),\fu(W)_1\right)$. 
\begin{Def}
We say the quintuple 
\[
(\xi,\al,\be)=(U(V_a)\times U(V_b), s, \eta,\al,\be)
\] is a \emph{relative elliptic endoscopic datum} and the direct sum of the $(-1)$-eigenspaces (\ref{endo spaces}) is an \textbf{endoscopic infinitesimal symmetric pair} for $\left(U(W_1)\times U(W_2),\fu(W)_1\right)$.
\end{Def}

With this definition, we show how to match regular semi-simple orbits and define the transfer of orbits with appropriate transfer factors in Section \ref{Section: relative endoscopy}. The key point is that we may relate the action of $U(W_1)\times U(W_2)$ on $\fu(W)_1$ to the adjoint action of $U(W_1)$ on its \emph{twisted Lie algebra} of Hermitian operators
\[
\Herm(W_1)=\{x\in \End(W_1): \la xu,v\ra=\la u, xv\ra\}.
\]
This is Proposition \ref{Lem: cat quotient Lie algebra 1}, realizing $\Herm(W_1)$ as the categorical quotient $\fu(W)_1//U(W_2)$.

Equipped with the above definition, we define smooth transfer of smooth compactly-supported test function as follows. Fix a regular semi-simple element $x\in \fu(W)_1$ and let $(\xi,\al,\be)$ be a relative endoscopic datum. This datum determines a character $\kappa$ we use to define the \emph{relative $\kappa$-orbital integral}
\begin{equation*}
\RO^\kappa(x,f) = \sum_{x'\sim_{st} x}\kappa(\inv(x,x'))\RO(x',f),
\end{equation*}
where $x'$ ranges over the $U(W_1)\times U(W_2)$-orbits in $\fu(W)_1$ in the same stable orbit as $x$ and the relative orbital integrals are as in Definition \ref{eqn: orbital int def}. When $\ka=1$, we write $\SRO=\RO^\ka$ and call this the relative stable orbital integral.

There is a good notion of when $x$ \textbf{matches} the pair $(x_a,x_b)\in \fu(V_a\oplus V_\al)_1\oplus\fu(V_b\oplus V_\be)_1$, and for such matching elements $(x_a,x_b)$ and $x$, we define the relative transfer factor
\begin{equation*}
\De_{rel}((x_a,x_b),x)
\end{equation*} in Section \ref{Section: relative endoscopy}.

\begin{Def}\label{Def: smooth transfer}We say $f\in C_c^\infty(\fu(W)_1)$ and $f_{\al,\be}\in C_c^\infty(\fu(V_a\oplus V_\al)_1\oplus\fu(V_b\oplus V_\be)_1)$ are \textbf{smooth transfers }(or say that they \textbf{match}) if the following conditions are satisfied:
\begin{enumerate}
\item For any matching orbits $x$ and $(x_a,x_b)$, we have an identify
\begin{equation}\label{eqn: transfer 1}
\SRO((x_a,x_b),f_{\al,\be})= \Delta_{rel}((x_a,x_b),x)\RO^\kappa(x,f).
\end{equation}
\item If there does not exist $x$ matching $(x_a,x_b)$, then
\begin{equation*}
\SRO((x_a,x_b),f_{\al,\be})= 0.
\end{equation*}
\end{enumerate}
\end{Def}

With this definition, we state Conjecture \ref{Conj: transfer exists}, asserting that smooth transfers exist for all smooth compactly-supported functions on $\fu(W)_1$. As a first check for our definition, we prove this conjecture for test functions supported in a certain open dense subset of $\fu(W)_1$ (see Proposition \ref{Prop: regular transfer}). 
\begin{Prop}
Let $f\in C^\infty_c(\fu(W)_1)$ and assume $\mathrm{supp}(f)$ is contained in the non-singular locus $\fu(W)_1^{iso}$ (see (\ref{regular locus}) below). Let $(U(V_a)\times U(V_b), s, \eta,\al,\be)$ be a {relative elliptic endoscopic datum}. 
 Then there exists $$f_{\al,\be}\in  C^\infty_c(\fu(V_a\oplus V_\al)_1\oplus\fu(V_b\oplus V_\be)_1)$$ such that $f$ and $f_{\al,\be}$ are smooth transfers of each other. 
\end{Prop}
Our proof of this proposition relies on the good behavior of the categorical quotient $$\fu(W)_1\lra \fu(W)_1//U(W_2)\cong \Herm(W_1)$$ over the non-singular locus. This enables a reduction to the analogous setting of the twisted Lie algebra. The general conjecture will require other techniques.

When $E/F$ is an unramified extension of non-archimedean local fields, we also formulate the fundamental lemma for the ``unit element.'' More specifically, suppose that $V_d=W_1=W_2$ is split, and let $\Lam_d\subset V_d$ be a self-dual lattice. In this case,
\[
\fu(W)_1= \Hom_E(V_d,V_d) =\End(V_d),
\] 
and the choice of self-dual lattice $\Lam_d$ gives a natural compact open subring $\End(\Lam_d)\subset \End(V_d)$. Let $\bfun_{\End(\Lam_d)}$ denote the indicator function for this subring. This also induces an integral model $\mathbf{U}(\Lam_d)$ of $U(V)$. Setting $U(\Lam_d)\subset U(V)$ as the $\calo_F$-points, this gives a hyperspecial maximal compact subgroup. 

 Now suppose that $\xi=(U(V_a)\times U(V_b),s,\eta)$ is an endoscopic datum for $\Herm(V_d)$. Under our assumptions, we have $V_d\cong V_a\oplus V_b$ (see Lemma \ref{Lem: unramified isom}). We may assume further that $\Lam_d=\Lam_a\oplus \Lam_b$ for fixed self-dual lattices $\Lam_a\subset V_a$ and $\Lam_b\subset  V_b$.  There are only four possible pairs $(\al,\be)$, and we set $(\al_0,\be_0)$ to be the split pair.
 
 We equip these groups with Haar measures normalized so that the given hyperspecial maximal subgroups have volume $1$.
\begin{Thm}\label{Conj: full fundamental lemmaintro}
The function $\bfun_{\End(\Lam_d)}$ matches $\bfun_{\End(\Lam_a)}\otimes\bfun_{\End(\Lam_b)}$ if $(\al,\be) = (\al_0,\be_0)$ and matches $0$ otherwise.
\end{Thm}
This is proved in \cite{leslieFL} via a combination of new local results on orbital integrals and a new global comparison or relative trace formulas. In Section \ref{Section: rank 2}, we verify this statement directly for $(U(4),U(2)\times U(2))$ by reducing to a family of transfer statements on the twisted Lie algebra and explicitly computing all orbital integrals involved.

We expect that this fundamental lemma implies the smooth transfer conjecture \ref{Conj: transfer exists}. This is analogous to the work of Waldspurger \cite{Waldspurgerimplies} for the Arthur-Selberg trace formula and Chong Zhang \cite{chongzhang} for the Guo-Jacquet relative trace formula. This reduction is work-in-progress by the author.


The outline of the paper is as follows. In Section \ref{Section: unitary endoscopy}, we recall the necessary notions and details from the theory of endoscopy, focusing on the case of the unitary Lie algebra. In Section \ref{Section: orbital integrals}, we define the symmetric space under consideration and define the orbital integral to be studied. In Section \ref{Section: relative endoscopy}, we define our proposal for a theory of relative endoscopy data in this setting and state the associated transfer conjecture and fundamental lemma. We prove the existence of transfer for many functions in Proposition \ref{Prop: regular transfer}. In Section \ref{Section: rank 2}, we end by proving this fundamental lemma for the case of $(U(4),U(2)\times U(2))$ by an explicit computation. 

\subsection{Acknowledgements}
We wish to thank Wei Zhang for sharing personal computations relating to the stable comparison outlined in \cite{youtube} (in particular, for the computation in Proposition \ref{Lem: standard pushforward}) which led us to the conjectured notion of endoscopic spaces and for general advice and encouragement. We also thank Jayce Getz for suggesting we consider the notion of relative endoscopy, as well as for many useful conversations and invaluable advice. Finally, we thank the anonymous referee for several helpful comments.
\subsection{Notation}
Throughout we assume that $E/F$ is a quadratic extension of local fields. When non-archimedean, we assume $F$ has either odd or zero characteristic. Let $\calo\subset F$ denote the ring of integers and let $\calo_E\subset E$ be its integral closure in $E$. We denote by $\val:E^\times\to \zz$ the unique extension of the normalized valuation on $F$. Thus, if $\vp$ is a uniformizer of $F$, then $\val(\vp)=1$. 

We fix an algebraic closure $\Fbar$ and a separable closure $F^{sep}$ of $F$ and let $\Ga=\Gal(F^{sep}/F)$ denote the Galois group. Denote by $\omega_{E/F}: F^\times \to \cc^\times$ the quadratic character associated to the extension $E/F$ via local class field theory. Let $\Nm =\Nm_{E/F}$ denote the norm map and set $U(1)=\ker(\Nm)$.

 We only consider smooth affine algebraic varieties over $F$. We use boldface notation for an algebraic variety $\mathbf{Y}$ and use Roman font $Y=\textbf{Y}(F)$ for its $F$-points. This space is naturally endowed with a locally-compact topology. When $F$ is non-archimedean, this topology makes $Y$ an $l$-space (see \cite{BZ2}), and we will consider the Schwartz space $C_c^\infty(Y)$ of locally-constant, compactly-supported $\cc$-valued functions. 

When $(W,\la\cdot,\cdot\ra)$ is a Hermitian space over $E$, we denote by $U(W):=U(W,\la\cdot,\cdot\ra)$ the associated unitary group. We set $V_d$ to be a fixed split Hermitian space of dimension $d$, so that $U(V_d)$ is a fixed quasi-split unitary group of rank $d$. We also fix representatives $\{\tau\}$ of the isomorphism classes of Hermitian form on the underlying vector space $V_d$, and denote by $V_\tau$ the associated pure inner form with $U(V_\tau)$ the unitary group. 

Any unitary group $U(W)$ acts on its Lie algebra $\fu(W)$ as well as its \emph{twisted} Lie algebra
\[
\Herm(W)=\{x\in \End(W): \la xu,v\ra=\la u, xv\ra\}
\]
by the adjoint action. For any $\de\in \Herm(W)$, we denote by $T_\de\subset U(W)$ the centralizer. 

\section{Endoscopy}\label{Section: unitary endoscopy}
In this section, we recall the necessary facts from the theory of endoscopy for unitary Lie algebras. We follow \cite{Xiao}, to which we refer the reader interested in proofs. Let $W$ be a $d$ dimensional Hermitian space over $E$. As previously noted, we will work with the twisted Lie algebra
\[
\Herm(W)=\{x\in \End(W): \la xu,v\ra=\la u, xv\ra\}.
\]

Let $\de\in \Herm(W)$ be regular and semi-simple. The theory of rational canonical forms implies that there is a decomposition $F[\de]:=F[X]/(\chi_\de(x))=\prod_{i=1}^mF_i$, where $F_i/F$ is a field extension and $\chi_\de(x)$ is the characteristic polynomial of $\de$. Setting $E_i=E\otimes_F F_i$, we have
\[
E[\de]=\prod_i E_i=\prod_{i\in S_1}E_i\times\prod_{i\in S_2}F_i\oplus F_i,
\]
where $S_1=\{i: F_i\nsupseteq E\}$ and $S_2 =\{i: F_i\supseteq E\}$.
\begin{Lem}\label{Lem: centralizers}
Let $\de\in \Herm(W)$ be regular semi-simple, let $T_\de$ denote the centralizer of $\de$ in $U(W)$. Then 
\[
T_\de\cong  Z_{\U(W)}(F)E[\de]^\times/F[\de]^\times,
\]
where $Z_{U(W)}(F)$ denotes the center of $U(W)$. Moreover, $H^1(F,T_\de)=\prod_{S_1}\zz/2\zz$ and 
\[
\calc(T_\de/F):=\ker\left(H^1(F,T_\de)\to H^1_{ab}(F,U(W))\right)=\ker\left(\prod_{S_1}\zz/2\zz\to \zz/2\zz\right),
\]
where $H_{ab}^1$ denotes abelianized cohomology in the sense of \cite{Borovoi} and the map on cohomology is the summation of the factors.
\end{Lem}
\begin{proof}
This is proved, for example, in \cite[3.4]{RogawskiBook}.
\end{proof}

Set $T_{S_1}\cong Z_{\U(W)}(F)\prod_{i\in S_1}E_i^\times/F_i^\times$ for the unique maximal compact subgroup of $T_\de$. We choose the unique measure $dt$ on $T_\de$ giving this subgroup volume $1$.

As we are studying the geometric theory of endoscopy (that is, $\ka$-orbital integrals), it is natural to use abelianized cohomology as in the lemma. If we define
\[
\cald(T_\de/F):=\ker\left(H^1(F,T_\de)\to H^1(F,U(W))\right),
\]
it is well known that the set of rational conjugacy classes $\calo_{st}(\de)$ in the stable conjugacy class of $y$ are in natural bijection with $\cald(T_\de/F)$. This is given by the map
\begin{equation}\label{invariant cohom}
\inv(\de,-):\calo_{st}(\de)\iso \cald(T_\de/F),
\end{equation} where
\[
[\de']\mapsto \inv(\de,\de'):=[\sig\in \Gal(E/F)\mapsto g^{-1}\sig(g)], 
\]
for any $g\in \GL(W)$ such that $\de'=\Ad(g)(\de).$ There is always an injective map
\[
\cald(T_\de/F)\lra \calc(T_\de/F),
\]
induced by the (surjective) map $H^1(F,U(W))\to H_{ab}^1(F,U(W))$. This latter map is an isomorphism when $F$ is a non-archimedean local field. In particular,  $\cald(T_\de/F)\cong \calc(T_\de/F)$ in this case.

When $F=\rr,$ the pointed set $\cald(T_\de/F)$ need not even be a group and the difference between the two is important. Subtle notions such as $K$-groups have been introduced to, among other things, understand spectral implications of this distinction. See \cite{LabesseSnowbird} for more details.


\subsection{Endoscopy for unitary Lie algebras}\label{Section: endoscopy roundup}
The general construction of elliptic endoscopic data and matching of stable orbits is standard, though rather involved. See \cite{RogawskiBook} for a review. Happily, our present case allows for a simplified presentation, enabling us to avoid several technicalities.

An elliptic endoscopic datum for $\Herm(W)$ is the same as a datum for the group $U(W),$ namely a triple $$(U(V_a)\times U(V_b),s,\eta),$$  where $a+b=d$. Here $s\in \hat{U}(W)$ a semi-simple element of the Langlands dual group of $U(W)$, and an embedding $$\eta:\hat{U}(V_a)\times\hat{U}(V_b)\hra \hat{U}(W)$$ identifying $\hat{U}(V_a)\times\hat{U}(V_b)$ with the neutral component of the centralizer of $s$ in the $L$-group ${}^LU(W)$. Fixing such a datum, we consider the endoscopic Lie algebra $$\Herm(V_a)\oplus \Herm(V_b).$$ Let $\de\in \Herm(W)$ and $(\de_a,\de_b)\in\Herm(V_a)\oplus \Herm(V_b)$ be regular semi-simple. 

Denote $W_{a,b}=V_a\oplus V_b$. In the non-archimedean case, the isomorphism class of $W_{a,b}$ is uniquely determined by those of $V_a$ and $V_b$ \cite[Theorem 3.1.1]{Jacobowitz}. For $F=\rr$, we can and do choose the form on $W_{a,b}$ to have signature $(n,n)$ or $(n+1,n)$, depending on the parity of $d$. In particular, $U(W_{a,b})$ is always quasi-split. 

\subsection{Matching of orbits}We first recall the notion of Jacquet--Langlands transfer between two non-isomorphic Hermitian spaces $W$ and $W'$. If we identify the underlying vector spaces (but not necessarily the Hermitian structures)
\begin{equation}\label{eqn: iso not Herm}
  W\cong E^n\cong W',  
\end{equation}
we have embeddings 
\[
\Herm(W),\:\Herm(W')\hra\fgl_n(E).
\]
Then $\de\in \Herm(W)$ and $\de'\in \Herm(W')$ are said to be \emph{Jacquet--Langlands transfers} if they are $\GL_n(E)$-conjugate in $\fgl_n(E).$  This is well defined since the above embeddings are determined up to $\GL_n(E)$-conjugacy. Note that if $\de$ and $\de'$ are Jacquet--Langlands transfers, then
\[
\de'=\Ad(g)(\de)
\]
for some $g\in \GL(W)$ and we obtain a well-defined cohomology class
\[
\inv(\de,\de'):=[\sig\in \Gal(E/F)\mapsto g^{-1}\sig(g)]\in H^1(F,T_\de)
\]
extending the invariant map to $\cald(T_\de/F)$.

\begin{Def}\label{Def: endoscopic matching}
In the case that $W'=W_{a,b}$, we say that $y$ and $(\de_a,\de_b)$ are \textbf{transfers (or are said to match)} if they are Jacquet--Langlands transfers in the above sense.
\end{Def}

For later purposes, note that we have an embedding
 $$\phi_{a,b}:\Herm(V_a)\oplus\Herm(V_b)\hra \Herm(W_{a,b}),$$ well defined up to conjugation by $U(W_{a,b})$. If $W\cong W_{a,b}$, we say that a matching pair $y$ and $(\de_a,\de_b)$ are a \textbf{\emph{nice matching pair}} if we may choose $\phi_{a,b}$ so that
 \[
 \phi_{a,b}(\de_a,\de_b) = \de.
 \]

\subsubsection{Orbital integrals}\label{Section:endoscopy character}
For $\de\in \Herm(W)^{rss}$ and $f\in C_c^\infty(\Herm(W))$, we define the orbital integral
\[
\Orb(\de,f)=\int_{T_\de\backslash U(W)}f(g^{-1}\de g)d\dot{g},
\]
where $dg$ is a Haar measure on $U(W)$, $dt$ is the unique normalized Haar measure on the torus $T_\de$, and $d\dot{g}$ is the invariant measure such that $dt d\dot{g}=dg.$

To an elliptic endoscopic datum $(U(V_a)\times U(V_b),s,\eta)$ and regular semi-simple element $\de\in \Herm(W)$, there is a natural character 
\[
\ka:\calc(T_\de/F)\to \cc^\times,
\]
which may be computed as follows. For matching elements $\de$ and $(\de_a,\de_b)$,
\begin{equation}\label{eqn: cohom decomp}
H^1(F,T_\de)=\prod_{S_1}\zz/2\zz=\prod_{S_1(a)}\zz/2\zz\times\prod_{S_1(b)}\zz/2\zz=H^1(F,T_{\de_a}\times T_{\de_b}),
\end{equation}
where the notation indicates which elements of $S_1$ arise from the torus $T_{\de_a}$ or $T_{\de_b}$.
\begin{Lem}\cite[Proposition 3.10]{Xiao}\label{Lem: endo character}
 Consider the character $\tilde{\ka}: H^1(F,T_\de)\to \cc^\times$ such that on each $\zz/2\zz$-factor arising from $S_1(a)$, $\tilde{\ka}$ is the trivial map, while it is the unique non-trivial map on each $\zz/2\zz$-factor arising from $S_1(b)$. Then $$\ka=\tilde{\ka}|_{\calc(T_\de/F)}.$$ 
\end{Lem}

Using the invariant map
\begin{equation*}
\inv(\de,-):\calo_{st}(\de)\iso \cald(T_\de/F)\hra \calc(T_\de/F),
\end{equation*}we form the $\ka$-orbital integral of $f\in C_c^\infty(\Herm(W_1))$
\[
\Orb^\ka(\de,f) = \sum_{\de'\sim_{st}\de}\ka(\inv(\de,\de'))\Orb(\de',f).
\]
When $\ka=1$ is trivial, write $\SO=\Orb^\ka.$
\subsubsection{Transfer factors}\label{Section: transfer factor}
We now recall the \emph{transfer factor} of Langlands-Sheldstad and Kottwitz. This is a function
\[
\De:[\Herm(V_a)\oplus \Herm(V_b)]^{rss}\times  \Herm(W)^{rss}\to \cc.
\]
The two important properties are
\begin{enumerate}
    \item $\De((\de_a,\de_b),\de) = 0$ if $\de$ does not match $(\de_a,\de_b),$ and
    \item if $\de$ is stably conjugate to $\de'$, then 
    \[
    \De((\de_a,\de_b),\de)\Orb^\ka(\de,f) = \De((\de_a,\de_b),\de')\Orb^\ka(\de',f).
    \]
\end{enumerate}
While the general definition, given in \cite{LanglandsShelstad1} for the group case and \cite{kottwitztransfer} in the quasi-split Lie algebra setting, is subtle, our present setting enjoys the following simplified formulation. 



When $\de\in \Herm(W)$ and $(\de_a,\de_b)\in \Herm(V_a)\oplus \Herm(V_b)$ do not match, we set
\[
\De((\de_a,\de_b),\de)=0.
\]
Now suppose that $\de$ and $(\de_a,\de_b)$ match. We define the relative discriminant
\[
D(\de)=\prod_{x_a,x_b}(x_a-x_b),
\]
where $x_a$ (resp. $x_b$) ranges over the eigenvalues of $\de_a$ (resp. $\de_b$) in $\Fbar$. 
\begin{Rem}
This is precisely the quotient of the standard Weyl discriminants that occurs in the factor $\De_{IV}$ in \cite{LanglandsShelstad1}. It is well known that the magnitude of $D(\de)$ measures the difference in asymptotic behavior of orbital integrals on $\Herm(W)$ and $\Herm(V_a)\oplus \Herm(V_b)$.
\end{Rem}
Recall our notation $W_{a,b}=V_a\oplus V_b$ and first assume that $W\cong W_{a,b}$ and that $\de$ and $(\de_a,\de_b)$ are a nice matching pair. In this case, the transfer factor is then given by
\begin{equation}\label{nice matching}
\De((\de_a,\de_b),\de):=\omega_{E/F}(D(\de))|D(\de)|_F,
\end{equation}
where $\omega_{E/F}$ is the quadratic character associated to $E/F$.

Now for any matching pair $\de$ and $(\de_a,\de_b)$, let
\[
\de'=\phi_{a,b}(\de_a,\de_b)\in \Herm(W_{a,b}).
\]
As discussed in Section \ref{Section: endoscopy roundup}, $\de$ and $\de'$ are Jacquet--Langlands transfers of each other and we set
\[
\De((\de_a,\de_b),\de) = \kappa(\inv(\de,\de'))\omega_{E/F}(D(\de))|D(\de)|_F,
\]
where $\ka:H^1(F,T_\de)\to \cc^\times$ is the character arising from the datum $(U(V_a)\times U(V_b),s,\eta)$ and $\inv$ is the extension of the invariant map discussed in Section \ref{Section: endoscopy roundup}. 
\begin{Rem}
For the interested reader, when $F$ is non-archimedean and $U(W)$ is quasi-split, our transfer factor agrees with the Lie algebra transfer factor studied in \cite{kottwitztransfer}, multiplied by the discriminant factor $\De_{IV}$. This formulation simply chooses a different distinguished conjugacy class in the stable orbit. See Remark \ref{Rem: kostant section}.
\end{Rem}

\subsection{Smooth transfer}
A pair of functions 
\[f\in C_c^\infty(\Herm(W))\:\text{ and }\:f_{a,b}\in C^\infty_c(\Herm(V_a)\oplus \Herm(V_b))
\]
are said to be smooth transfers (or matching functions) if the following conditions are satisfied:
\begin{enumerate}
\item for any matching elements regular semi-simple elements $\de$ and $(\de_a,\de_b)$,
\begin{equation*}
\SO((\de_a,\de_b),f_{a,b})= \Delta((\de_a,\de_b),\de)\Orb^\kappa(\de,f);
\end{equation*}
\item if there does not exist $y$ matching $(\de_a,\de_b)$, then
\begin{equation*}
\SO((\de_a,\de_b),f_{a,b})= 0.
\end{equation*}
\end{enumerate}
The following theorem follows by combining \cite{LaumonNgo}, \cite{WaldsCharacteristic}, and \cite{Waldstransfert}.
\begin{Thm}\label{Thm: smooth transfer lie algebra}
For any $f\in C_c^\infty(\Herm(W_1))$, there exists a smooth transfer $f_{a,b}\in C^\infty_c(\Herm(V_a)\oplus \Herm(V_b))$.
\end{Thm}


\section{The Lie algebra of the symmetric space}\label{Section: orbital integrals}
Let $E/F$ be a quadratic extension of local fields of odd residue characteristic. Let $W_1$ and $W_2$ be two Hermitian spaces of dimension $d$ over $E$. Let $\fu(W)$ denote the Lie algebra of $U(W)$, where $W=W_1\oplus W_2$ is a $2d$ dimensional Hermitian space. The differential of the involution $\sig$ acts on $\fu(W)$ by the same action and induces a $\zz/2\zz$-grading
\[
\fu(W)= \fu(W)_0\oplus \fu(W)_1,
\]
where $\fu(W)_i$ is the $(-1)^i$-eigenspace of the map $\sig$. 
\begin{Lem}\label{make life easy}
We have natural identifications 
\[
\fu(W)_0=\fu(W_1)\oplus \fu(W_2),\text{  and   }\fu(W)_1=\Hom_{E}(W_2,W_1).
\]
Here $U(W_1)\times U(W_2)$ acts on $\fu(W)_1$ by the restriction of the adjoint action. In terms of $W_1$ and $W_2$, the action is given by $(g,h)\cdot \varphi = g\circ \varphi \circ h^{-1}$. 
\end{Lem}
\begin{proof}
If $\sigma: W\to W$ denotes the linear involution, the involution induced on $\fu(W)$ is $x\mapsto \sigma\circ x\circ\sigma$. It is a simple exercise in the definitions that any element $x\in \fu(W)$ may be uniquely expressed as 
\[x= \left(\begin{array}{cc}x_{11}&x_{12}\\x^\ast_{12}&x_{22}\end{array}\right),
\]
where $x_{ii}\in \fu(V_i)$, $x_{12}\in \Hom(W_2,W_1)$ and if $\la\cdot,\cdot\ra_i$ denotes the Hermitian pairing on $V_i$, then $x^\ast_{12}\in \Hom(W_1,W_2)$ is the unique linear map satisfying
\[
\la x_{12}v,w\ra_1 = \la v,x^\ast_{12}w\ra_2
\] for all $v\in W_2$ and $w\in W_1$. It follows that
\[
\sigma(x) = \left(\begin{array}{cc}x_{11}&-x_{12}\\-x_{12}^\ast&x_{22}\end{array}\right),
\]
and the lemma follows.
\end{proof}
In particular, any element $x\in \fu(W)_1$ may be uniquely written
\[
x=\de(X)= \left(\begin{array}{cc}&X\\-X^\ast&\end{array}\right),
\]
where $X\in \Hom_{E}(W_2,W_1)$. For any such $x$, we denote by $H_x\subset U(W_1)\times U(W_2)$ the stabilizer of $x$.

Define the regular semi-simple locus $\fu(W)_1^{rss}$ to be the set of $\de\in \fu(W)_1$ whose orbit under $U(W_1)\times U(W_2)$ is closed and of maximal dimension. In our present case, we have 
\[\fu(W)_1^{rss}=\fu(W)_1\cap \fu(W)^{rss},\] where $\fu(W)^{rss}$ is the classical regular semi-simple locus of the Lie algebra. This is due to the fact that the symmetric pair $(U(W),U(W_1)\times U(W_2)$ is geometrically quasi-split. See \cite[Sec. 2]{Lesliespringer} for more details on quasi-split symmetric spaces. 

 Let 
\begin{equation}\label{regular locus}
\fu(W)_1^{iso}\cong\mathrm{Iso}_E(W_2,W_1)
\end{equation} be the open subvariety of elements $\de(X)$ where $X:W_2\to W_1$ is a linear isomorphism. We refer to $\fu(W)_1^{iso}$ as the \emph{non-singular locus}.
There are natural contraction maps $r_i:\fu(W)_1\to \Herm(W_i)$ given by 
\[
r_i(\de(X))= \begin{cases} -XX^\ast :\quad i=1\\ -X^\ast X:\quad i=2.\end{cases}
\]
Define the map $\pi:\fu(W)_1\to \A^n$ given by $\pi(x) = (a_1(x),\ldots,a_n(x))$, where
\[
a_i(x)=\text{the coefficient of $t^{i-1}$ in}\quad \det(tI-r_1(x)).
\] 
\begin{Lem}\label{Lem: cat quotient Lie algebra 1}
The map $r:= r_1$ intertwines the $U(W_1)$ action on $\fu(W)_1$ and the adjoint action on $\Herm(W_1)$. Moreover, the pair $(\Herm(W_1), r)$ is a categorical quotient for the $U(W_2)$-action on $\fu(W)_1$.
\end{Lem}
\begin{proof} The equivariance statement is obvious. 
As the categorical quotient assertion is geometric, we may assume without loss that $F=\Fbar$. The action we consider is following action of $\GL_d(F)\times \GL_d(F)$ on $Mat_d(F)\times Mat_d(F)$:
\[
(g,h)\cdot (X,Y) = (gXh^{-1},hYg^{-1}).
\]
The map $r$ becomes the product map
\begin{align*}
Mat_d(F)\times Mat_d(F)&\to Mat_d(F)\\
					(X,Y)&\mapsto  XY.
\end{align*}We make use of Igusa's criterion \cite[Section 3]{ZhangFourier}: let a reductive group $H$ act on an irreducible affine variety $X$. Let $Q$ be a normal irreducible variety, and let $\pi:X\to Q$ be a morphism that is constant on $H$ orbits such that
\begin{enumerate}
\item $Q-\pi(X)$ has codimension at least two,
\item there exists a nonempty open subset $Q'\subset Q$ such that the fiber $\pi^{-1}(q)$ of $q\in Q'$ contains exactly one orbit.
\end{enumerate}
Then $(Q,\pi)$ is a categorical quotient of $(H,X)$.  Note that it is clear that $r$ is surjective as $X\to (X,I_d)$ provides a section, so that the first criterion is satisfied. For the second criterion, we note that the open set $Q'=\GL_d(F)$ works.
\end{proof}

Note that a similar argument gives the following lemma for the quotient by both unitary actions.

\begin{Lem}\label{Lem: cat quotient Lie algebra}
The pair $(\A^d,\pi)$ is a categorical quotient for the $U(W_1)\times U(W_2)$ action on $\fu(W)_1$.
\end{Lem}\begin{proof}
As in the proof of the previous proposition, we may pass to the algebraic closure, at which point it is evident that the map $\pi$ is surjective. The uniqueness of orbits over a non-empty subset follows from the associated statement in Proposition \ref{Lem: cat quotient Lie algebra 1} and the theory of rational canonical forms.
\end{proof}

\begin{Lem}\label{Lem: regular semi-simple locus} There is an inclusion $\fu(W)_1^{rss}\subset\fu(W)_1^{iso}$.
\end{Lem}
\begin{proof}We again pass to the algebraic closure $F=\Fbar$ and assume that $\fu(W)\cong \fgl_{2d}(F)$. Just as before, we now consider the action of $\GL_d(F)\times \GL_d(F)$ on $Mat_d(F)\times Mat_d(F)$ by $(g,h)\cdot (X,Y) = (gXh^{-1},hYg^{-1}).$ As before, we are now considering the action
\[
(g,h)\cdot (X,Y) = (gXh^{-1},hYg^{-1}).
\]
of $\GL_d(F)\times \GL_d(F)$ on $Mat_d(F)\times Mat_d(F)$. The invariant of this action is $\pi(X,Y)(t)=\det(tI-XY)$ as in Lemma \ref{Lem: cat quotient Lie algebra}. 

Recalling that the infinitesimal symmetric space $Mat_d(F)\times Mat_d(F)$ is quasi-split, the element $(X,Y)$ is regular semi-simple in $Mat_d(F)\times Mat_d(F)=\fgl_{2d}(F)_1$ if and only if the element
\[
Z=\left(\begin{array}{cc}&X\\Y&\end{array}\right)\in \fgl_{2d}(F)
\] is regular semi-simple. Letting $\chi_Z(t)=\det(tI-Z)$ denote the characteristic polynomial, $Z$ is regular semi-simple if and only if $\chi_Z$ has distinct roots. Now a simple exercise in linear algebra shows that
\[
\chi_Z(t)=\pi(X,Y)(t^2).
\]
Thus, $\ga\in \fgl_{2d}(F)^{rss}$ is possible only if $0$ is not a root of $\pi(X,Y)$, implying the lemma.
\end{proof}
 \subsection{Relative orbital integrals} We now introduce the primary objects of interest: the relative orbital integrals for the symmetric pair $(U(W_1)\times U(W_2),\fu(W)_1)$. For any $x\in \fu(W)_1$, we set
\[
H_x:=\{(h,g)\in U(W_1)\times U(W_2): h^{-1}xg=x\}.
\]
\begin{Def}
For $f\in C_c^\infty(\fu(W)_1)$, and $x\in \fu(W)_1$ a relatively semi-simple element, we define the \textbf{relative orbital integral} of $f$ by
\begin{equation}\label{eqn: orbital int def}
\RO(x,f) := \displaystyle\iint_{H_x\backslash U(W_1)\times U(W_2)}f(h_1^{-1}x h_2) {d\dot{h}_1d\dot{h}_2},
\end{equation}
where $dh_i$ are Haar measures on $U(V_i)$ and $d\dot{h}_1d\dot{h}_2$ is the invariant measure such that $dt d\dot{h}_1d\dot{h}_2 = d{h}_1d{h}_2$. Here $dt$ the unique normalized measure on the torus $H_x$ as in Section \ref{Section: unitary endoscopy}. As always, the value of $\RO(f,x)$ depends on the choice of the measures $dh_i$.
\end{Def}
We note that since the orbit of $x$ is closed, the integral is absolutely convergent. 
Let 
\[
\Herm(W_1)^{iso}:=\Herm(W_1)\cap \GL(W_1)
\]be the open subset of non-singular Hermitian forms.
\begin{Lem}\label{Lem: centralizer contraction}
The restriction $r:\fu(W)_1^{iso}\to \Herm(W_1)^{iso}$ is a $U(W_2)$-torsor. Moreover, for $x\in \fu(W)^{iso}_1$, we have an isomorphism
\[
H_x\iso T_{r(x)}
\]
given by $(h_1,h_2)\mapsto h_1$.
\end{Lem}

\begin{proof}For the first claim, we saw in the proof of Proposition \ref{Lem: cat quotient Lie algebra 1} that the claim holds over the algebraic closure of $F$, which suffices. For the second claim, we construct an inverse. Let $h\in T_{r(x)}$. Then $hx$ also lies in the fiber over $r(x)$. By the torsor property, there exists a unique $h'\in U(W_2)$ such that $$hx=xh'.$$ Define the inverse morphism by $h\mapsto (h,h')$. It is clear that this gives a section.
\end{proof}
\begin{Notation}
We will always use lower-case Roman letters $x,y$ to denote vectors in the infinitesimal symmetric space $\fu(W)_1$ and the like, and will use lower-case Greek letters $\de,\ga$ to denote vectors in the Hermitian quotient $\Herm(W_1)$, etc.
\end{Notation}

\subsection{Application of the contraction map}
By Lemma \ref{Lem: centralizer contraction}, the contraction map to the non-singular locus of $\fu(W)_1^{iso}$ is a $U(W_2)$-torsor. Noting that 
\[
\fu(W)_1^{iso}\subset \fu(W)_1
\]
is $\GL(\fu(W)_1)$-stable, Corollary \ref{Cor: torsor rational} implies the decomposition
\begin{equation*}
\Herm(W_1)^{iso} = \bigsqcup_{[\al]\in H^1(F,U(W_2))} \fu(W)_1^{iso}/U(V_\al),
\end{equation*}
where the subscript $\al$ indicates the appropriate pure inner twist.
Proposition \ref{Prop: surjective schwartz} thus tells us that 
\[
\left(r/U(W_2)\right)_!:\bigsqcup_{[\al]\in H^1(F,U(W_2))} C_c^\infty(\fu(W)_1^{iso})\to C_c^\infty(\Herm(W_1)^{iso})
\]
is surjective. We may extend this to a non-surjective map \[\left(r/U(W_2)\right)_!:\bigsqcup_{[\al]\in H^1(F,U(W_2))} C_c^\infty(\fu(W)_1)\to C^\infty(\Herm(W_1)^{iso}),\]  where the push-forward $r_!(f)$ will not be compactly supported if $\supp(f)$ is not contained in $\fu(W)_1^{iso}$. It is always of relatively-compactly support in $\Herm(W_1)$ \cite[Lemma 3.12]{ZhangFourier}.  

\begin{Lem}\label{Lem: orbital reduction}
Fix Haar measures on $U(W_1)$, $U(W_2)$, and $H_x\cong T_{r(x)}$. For $f\in C_c^\infty(\fu(W)_1)$ and for $x\in \fu(W)_1$ regular semi-simple, we have $$\RO(x,f) = \Orb(r(x),r_!(f)),$$
where both sides are normalized with our choices of Haar measures.
\end{Lem}
\begin{proof}
If $x$ is a regular semi-simple element, then everything is absolutely convergent as the orbit of $r(x)$ is closed and the the support of $\supp(r_!(f))$ is relatively compact in $\Herm(W_1)$. Rearranging the integrals and applying Lemmas \ref{Lem: regular semi-simple locus} and \ref{Lem: centralizer contraction} implies that
\begin{equation*}\label{eqn: reduction 1}
\RO(x,f) = \displaystyle\int_{T_{r(x)}\backslash U(W_1)}r_!(f)(g^{-1}r(x)g)d\dot{g}.
\end{equation*}
\end{proof}

The following lemma is needed to study relative $\ka$-orbital integrals. Let $x\in \fu(W)_1^{rss}$ and set $\de=r(x)\in \Herm(W_1)^{rss}$.

\begin{Lem}\label{Lem: cohom iso cent}
Let $\phi:H_x\iso T_\de$ be the map from Lemma \ref{Lem: centralizer contraction}. Then $\phi$ induces an isomorphism between
\begin{equation*}\label{eqn: iso kernels}
\calc(H_x/F)\iso \calc(T_\de/F),
\end{equation*}
where $$\calc(H_x/F)=\ker\left(H^1(F,H_x)\to H^1_{ab}(F, U(W_1)\times U(W_2))\right)$$ and $$\calc(T_\de/F)=\ker\left(H^1(F,T_\de)\to H^1_{ab}(F,U(W_1))\right).$$ 
\end{Lem}
\begin{proof}
Consider the commutative diagram
\[
\begin{tikzcd}
H^1(F,H_x)\ar[r,"\iota_x"]\ar[d,"\phi^\ast"]& H^1_{ab}(F,U(W_1))\times H^1_{ab}(F,U(W_2))\ar[d,"p_1^\ast"]\\
H^1(F,T_\de)\ar[r,"\iota_\de"]&H^1_{ab}(F,U(W_1)).
\end{tikzcd}
\]
where $\phi^\ast$ and $p_1^\ast$ are the maps induce on cohomology.

If $\al\in\calc(H_x/F)$, then $\iota_\de\phi(\al)=p_1(\iota_x(\al))=1.$ This allows us to extend the diagram to
\[
\begin{tikzcd}
1\ar[r]&\calc(H_x/F)\ar[r]\ar[d]&H^1(F,H_x)\ar[r,"\iota_x"]\ar[d,"\phi^\ast"]& H^1_{ab}(F,U(W_1))\times H^1_{ab}(F,U(W_2))\ar[d,"p_1^\ast"]\\
1\ar[r]&\calc(T_\de/F)\ar[r]&H^1(F,T_\de)\ar[r,"\iota_\de"]&H^1_{ab}(F,U(W_1)),
\end{tikzcd}
\]
where the arrow $\calc(H_x/F)\to \calc(T_\de/F)$ is an injection. By Lemma \ref{Lem: centralizers}, we need only show that $|\calc(H_x/F)|=2^{|S_1|-1}$. But this follows since the image of $\iota_x$ lies in
\[
\ker\left(H^1_{ab}(F,U(W_1)\times U(W_2))\to H^1_{ab}(F,U(W))\right).
\]
Combined with the injection $\calc(H_x/F)\to \calc(T_\de/F)$, this forces
\[
2^{|S_1|-1}\leq |\calc(H_x/F)\leq 2^{|S_1|-1},
\]
and the lemma follows.
\end{proof}

\section{Relative endoscopy}\label{Section: relative endoscopy}
 In this section, we fill in the details of our notion of endoscopic symmetric spaces and study the conjectural transfer and the fundamental lemma in this context. We establish the existence of smooth transfer for many functions. 

\subsection{Relative endoscopic data}
Let $\xi=(U(V_a)\times U(V_b),s,\eta)$ be an elliptic endoscopic datum for $U(W_1)$, where $d=a+b$. Fix representatives $\{\al\}$ and $\{\be\}$ of the isomorphism classes of Hermitian forms on the underlying vector spaces of $V_a$ and $V_b$. For each pair $(\al,\be)$, we have the Lie algebras
\[
\fu(V_a\oplus V_\al)\text{   and    }\fu(V_b\oplus V_\be),
\] where $V_\al$ simply denotes the Hermitian space $(V_a,\al)$, and similarly for $V_\be$. Each equipped with a natural involution $\sig_\al$ and $\sig_\be$ inducing infinitesimal symmetric pairs
\[
\left(U(V_a)\times U(V_\al),\fu(V_a\oplus V_\al)_1\right) \text{ and }\left(U(V_b)\times U(V_\be),\fu(V_b\oplus V_\be)_1\right).
\]
  As in the introduction, we call the quintuple
\[
(\xi,\al,\be)=(U(V_a)\times U(V_b), s, \eta,\al,\be)
\] is a {relative elliptic endoscopic datum} and the direct sum of these infinitesimal symmetric pairs is an {endoscopic infinitesimal symmetric pair} for $\left(U(W_1)\times U(W_2),\fu(W)_1\right)$.

\begin{Notation}
For the next two subsections, we adopt the following notation:  fix a relative endoscopic datum $(\xi,\al,\be)$ and set
\[
\text{$\fg=\fu(W)$\: and\: $\fh^{\al,\be}=\fu(V_a\oplus V_\al)\oplus\fu(V_b\oplus V_\be)$.}
\]
\end{Notation} 
 
 The endoscopic space $\fh_1^{\al,\be}$ comes equipped with the $U(V_\al)\times U(V_\be)$-invariant contraction map
\[
r_{\al,\be}:\fh^{\al,\be}_1\to \Herm(V_a)\oplus \Herm(V_b).
\]

\begin{Def}
We say that $x\in \fg_1^{rss}$ \textbf{matches} the pair $(x_a,x_b)\in (\fh_1^{\al,\be})^{rss}$ if 
\[
\text{$r(x)\in \Herm(W_1)$ and $r_{\al,\be}(x_a,x_b)\in \Herm(V_a)\oplus\Herm(V_b)$}
\]match in the sense of Definition \ref{Def: endoscopic matching}.
\end{Def}
 For matching elements $(x_a,x_b)$ and $x$, we define the \textbf{relative transfer factor}
\begin{equation}\label{eqn: transfer factor}
\De_{rel}((x_a,x_b),x):=\De(r_{\al,\be}(x_a,x_b), r(x)),
\end{equation}
where the right-hand side is the definition given in Section \ref{Section: transfer factor}.
\begin{Lem}
As $(x_a,x_b)$ varies over a stable $(U(V_a)\times U(V_\al))\times(U(V_b)\times U(V_\be))$-orbit in $\fh^{\al,\be}_1$, the element $x$ varies over a stable $U(W_1)\times U(W_2)$-orbits in $\fg_1$.
\end{Lem}

\begin{proof}
This follows from the corresponding case for unitary Lie algebras and the fact that the regular stabilizers in the quotients $r:\fg_1\to \Herm(W_1)$ are trivial.
\end{proof}
\subsection{Smooth transfer}

 Fix $x\in \fg_1^{rss}$ and let $(\xi,\al,\be)$ be a relative endoscopic datum. The endoscopic triple $\xi=(H,s,\eta)$ of $U(W_1)$ determines a character 
 \[
 \kappa:\calc(T_{r(x)}/F)\to \cc^\times
 \]
 via the construction of Lemma \ref{Lem: endo character}. By Lemma \ref{Lem: cohom iso cent}, we may pull this character back along the isomorphism
 \[
 \calc(H_x/F)\iso \calc(T_{r(x)}/F),
 \]
to obtain a character which we also call $\ka: \calc(H_x/F) \to\cc^\times$. We now define the relative $\kappa$-orbital integral to be
\begin{equation*}
\RO^\kappa(x,f) := \sum_{x'\sim x}\kappa(\inv(x,x'))\RO(x',f),
\end{equation*}
where $x'$ ranges over the representatives of the $U(W_1)\times U(W_2)$-orbits stably conjugate to $x$. By Lemma \ref{Lem: cohom iso cent}, $\inv(x,x')=\inv(r(x),r(x'))$.

We stated the definition of smooth transfer in the context of relative endoscopy in Definition \ref{Def: smooth transfer}.

\begin{Conj}\label{Conj: transfer exists}
For any relative endoscopic datum $(\xi,\al,\be)$ and any  $f\in C_c^\infty(\fg_1)$, there exists $f_{\al,\be}\in C^\infty_c(\fh^{\al,\be}_1)$. such that $f$ and $f_{\al,\be}$ match.
\end{Conj}

We remark that this conjecture does not depend on the choice of Haar measures used in defining the orbital integrals; coherent choices are needed for the fundamental lemma in the next section. We do not prove this conjecture in general, but are able to establish it in many cases.
\begin{Prop}\label{Prop: regular transfer}
Let $f\in C^\infty_c(\fg_1)$ and assume $\mathrm{supp}(f)\subset \fg_1^{iso}$. 
 Then there exists $f_{\al,\be}\in  C^\infty_c(\fh^{\al,\be}_1)$ such that $f$ and $f_{\al,\be}$ match.
\end{Prop}

\begin{proof}
As we have seen, the restriction of the contraction map 
\[
r: \fg_1^{iso}\to \Herm(W_1)^{iso}
\]is a $U(W_2)$-torsor. In particular, it is a submersion onto its image, implying that if $\mathrm{supp}(f)\subset \fg_1^{iso}$,
\[
r_!(f)\in C^\infty_c(\Herm(W_1)^{iso})
\]
is also compactly supported. Setting $\de=r(x)$, we may now apply Theorem \ref{Thm: smooth transfer lie algebra}
 to find a smooth compactly-supported function $$f_{a,b}:\Herm(V_a)\oplus\Herm(V_b)\to \cc$$ matching $r_!(f)$.
 


 
 Since the open subset 
\[
\Herm(V_a)^{iso}\oplus\Herm(V_b)^{iso}\subset \Herm(V_a)\oplus\Herm(V_b)
\]
 is determined by the non-vanishing of the determinant, it follows that $\de\in \Herm(W_1)^{iso}$ if and only if $(\de_a,\de_b)\in \Herm(V_a)^{iso}\oplus\Herm(V_b)^{iso}$. In particular, 
 \[
 \SO((\de_a,\de_b),f_{a,b}) =\De((\de_a,\de_b),\de)\Orb^\ka(\de,r_!(f))=0
 \]
 whenever $\de\in \Herm(W_1)^{rss}$ does not lie in $\Herm(W_1)^{iso}$. We thus lose no orbital integral information by replacing $f_{a,b}$ with
 \[
f_{a,b}\cdot \left(\bfun_{\Herm(V_a)^{iso}}\otimes\bfun_{\Herm(V_b)^{iso}}\right).
 \] 
In particular, we can assume  $\supp(f_{a,b})\subset \Herm(V_a)^{iso}\oplus\Herm(V_b)^{iso}$. 

Considering the $U(V_\al)\times U(V_\be)$-torsor
\[
r_{\al,\be}:(\fh^{\al,\be}_1)^{iso}\lra\Herm(V_a)^{iso}\oplus\Herm(V_b)^{iso},
\]another application of Corollary \ref{Cor: torsor rational} implies that $\Herm(V_a)^{iso}\oplus\Herm(V_b)^{iso}$ decomposes as a disjoint union
\[
 \bigsqcup_{[\ep,\nu]\in H^1(F,U(V_a))\times H^1(F,U(V_b))} (\fh^{\al,\be}_1)^{iso}/U(V_\ep)\times U(V_\nu).
\]
 Proposition \ref{Prop: surjective schwartz} now implies that there exist functions $$f_{\ep,\nu}\in C^\infty_c((\fh^{\ep,\nu}_1)^{iso}),$$ with $\{(\ep,\nu)\}$ ranging over pure inner forms $H^1(F,U(V_a)\times U(V_b))$, such that $$f_{a,b} = \sum_{\ep,\nu}(r_{\ep,\nu})_!(f_{\ep,\nu}).$$
 In this way, for  $r_{\al,\be}(x_a,x_b) = (\de_a,\de_b)$ we find that 
\begin{equation*}
\SRO((x_a,x_b),f_{\al,\be})=\SO((\de_a,\de_b),f_{a,b}),
\end{equation*}
 proving the proposition.
\end{proof}

The main obstruction to proving Conjecture \ref{Conj: transfer exists} is that while $r_!(f)$ is always of relatively compact support. More information is needed about the singularities of this map to extend the previous result.

\subsection{The endoscopic fundamental lemma}
We now assume that $E/F$ is an unramified extension of non-archimedean local fields. Suppose that $V_d=W_1=W_2$ is split. In the non-archimedean setting, this implies that there exists a self-dual lattice $\Lam_d\subset V_d$, a choice of which we now fix. In this case,
\[
\fu(W)_1= \Hom_E(V_d,V_d) =\End(V_d),
\] 
and the choice of self-dual lattice $\Lam_d$ gives a natural compact open subring $\End(\Lam_d)\subset \End(V_d)$. Let $\bfun_{\End(\Lam_d)}$ denote the indicator function for this subring. This also induces an integral model $\mathbf{U}(\Lam_d)$ of $U(V)$. Setting $U(\Lam_d)\subset U(V)$ as the $\calo_F$-points, this gives a hyperspecial maximal compact subgroup.

 \begin{Lem}\label{Lem: unramified isom}
 Suppose that $\xi=(U(V_a)\times U(V_b),s,\eta)$ is an endoscopic datum for $\Herm(V_d)$. Under our assumptions, we have $V_d\cong V_a\oplus V_b$.
 \end{Lem} 
 \begin{proof}
 First, we recall the notion of a determinant $d(V)$ of a Hermitian space $V$: for any basis $\{x_\lam\}$ of $V$, set let $\det(\la x_\lam,x_\mu\ra)$ be the determinant of the resulting matrix representation of the Hermitian form. Change of basis multiplies this by a norm from $E^\times$, so we set $d(V)$ to be the  resulting class in $F^\times/\Nm_{E/F}(E^\times)$. The theorem of Jacobowitz \cite[Theorem 3.1.1]{Jacobowitz} tells us that two Hermitian spaces $V$ and $W$ are isomorphic if and only if
 \[
 \text{$\dim(V)=\dim(W)$ and $d(V)\sim d(W)$.}
 \]
 
 In the unramified non-archimedean setting, it is well known that a Hermitian space is split if and only if  $d(V)$ is a norm. For example, we can choose a basis such with respect to which the form is represented by the identity matrix. Since 
 \[
 d(V_a\oplus V_b) = d(V_a)\cdot d(V_b)
 \]
 is then also a norm, the lemma follows.
  \end{proof}
 We fix such an isomorphism by imposing $\Lam_d=\Lam_a\oplus \Lam_b$ for fixed self-dual lattices $\Lam_a\subset V_a$ and $\Lam_b\subset  V_b$; this is determined up to $U(\Lam_d)\times U(\Lam_d)$-conjugation.  Note that there are only four possible pairs $(\al,\be)$, and we set $(\al_0,\be_0)$ to be the split pair. We equip these groups with Haar measures normalized so that the given hyperspecial maximal subgroups have volume $1$.
\begin{Thm}\label{Conj: full fundamental lemma}(Relative fundamental lemma)
 If $(\al,\be) = (\al_0,\be_0)$, the functions $\bfun_{\End(\Lam_d)}$ and $\bfun_{\End(\Lam_a)}\otimes\bfun_{\End(\Lam_b)}$ are smooth transfers. Otherwise, $\bfun_{\End(\Lam_d)}$ matches $0$.
\end{Thm}

This theorem is the main result of \cite{leslieFL}, and its proof is beyond the scope of this paper. In the next section, we give a proof of this statement in the case of $(U(V_4),U(W_2)\times U(W_2))$. 

\section{The relative fundamental lemma for {$(U(V_4),U(V_2)\times U(V_2))$}}\label{Section: rank 2}
 We continue with the assumption that $E/F$ is an unramified extension of non-archimedean local fields. Let $U(V_4)$ be the quasi-split unitary group of rank $4$ and $(U(V_2)\times U(V_2), \End(V_2))$ the associated symmetric space. In this case, the only non-trivial endoscopic space to consider is $\End(V_1)\oplus \End(V_1)\cong E\oplus E$ with the action of $[U(V_1)\times U(V_1)]\times[U(V_1)\times U(V_1)]$. 

\begin{Thm}\label{Thm: fundamental lemma}
For the endoscopic space $\End(V_1)\oplus \End(V_1)$ of $\End(V_2)$, the fundamental lemma holds. 
\end{Thm}
Our proof is computational. Let $\Lam$ be our rank $2$ self-dual lattice, and let $\bfun_{\End(\Lam)}$ be the associated indicator function. The idea is to compute the push forward
\[
r_!(\bfun_{\End(\Lam)})(XX^\ast) = \int_{U(V_2)}\bfun_{\End(\Lam)}(Xh)dh.
\]
Once we have done this, we compute the associated integrals on the twisted Lie algebra and verify the $\kappa$-orbital integrals agree with the stable relative orbital integrals on the endoscopic side. The proof will be completed by combining Proposition \ref{Prop: orbital computation} and Proposition \ref{Prop: endoscopic side} below.

\subsection{Computing the pushforward}
For the sake of computation, we fix an element $\zeta$ such that $E=F(\zeta)$ where $\zeta\in \calo_E^\times$ and $\overline{\zeta}=-\zeta$. Here the overline indicates the non-trivial Galois element. We also fix the split Hermitian form
\[
J= \left(\begin{array}{cc}&\zeta\\-\zeta&\end{array}\right).
\]
The contraction morphism 
\[
r: \End(V_2)\lra \Herm(V_2)
\]
is given by $X\mapsto XX^\ast$, where $X^\ast = J\overline{X}^T J^{-1}$. 

Set $\Phi:=r_!\bfun_{\End(\Lam)}$. Then $\Phi$ is supported on the subset of $\Herm(V_2)$ with $\val(\det))\geq0$. Note that the group $\GL(V_2)$ acts on $\Herm(V_2)$ via the twisted action
\[
g\cdot \de = g\de g^{\ast}.
\]
and that for any $g\in \GL(\Lam)\subset \GL(V_2)$,
\[
\Phi(gr(x)g^\ast) = \int_{U(V_2)} \bfun_{\End(\Lam)}(gxh) dh=\int_{U(V_2)}\bfun_{\End(\Lam)}(gxh) dh=\Phi(r(x)).
\]
Thus, $\Phi$ is constant on $\GL(\Lam)$-orbits of $\Herm(V_2)$. By \cite{jacobowitz1962hermitian}, we may choose the forms
\[
 \varpi^{(i,j)}=\zeta\left(\begin{array}{cc}&\varpi^i\\\varpi^j&\end{array}\right)
\]
as representatives of these orbits.
\begin{Prop}\label{Lem: standard pushforward}
Let $\Phi_n$ be the restriction of $\Phi$ to $\Herm_{\val(\det)=n}$. Then $\Phi_n=0$ if $n$ is odd or $n<0$. If $n\geq0$ is even, we compute that
\[
\Phi_n=\sum_{k=0}^{n/2}\left(\sum_{j=0}^kq^j\right)\mathbf{1}_{\GL(\Lam)\cdot\varpi^{(k,n-k)}}=\sum_{k=0}^{n/2}q^k\bfun_{\Herm(\Lam)_{n,k}},
\]
where $\Herm(\Lam)_{n,k}:=\Herm(V_2)\cap \vp^k \End(\Lam)_{\val(\det)=n-2k}$.
\end{Prop}

\begin{proof}
As noted above, if $M\in Im(r)$, then $\val(\det(M))$ is even. This implies that $\Phi_n=0$ when $n$ is odd, so we assume now that $n$ is even.  Also, $\Phi_n=0$ for $n<0$. Finally, the equality of the two expressions for $\Phi_n$ is a simple exercise. We thus show the left-most expression.

We need only to compute $\Phi(\varpi^{(i,j)})$. Since $\Phi(\varpi^{(i,j)})= \Phi(\varpi^{(j,i)})$, we are free to assume that $i\leq j$. Noting that $j-i=2l$ is even, we choose a section of the invariant map $r$ over $\vp^{(i,j)}$:
\[
X_{(i,j)}=\left(\begin{array}{cc}1&\frac{\vp^i\zeta}{2}\\\vp^l&-\frac{\vp^{i+l}\zeta}{2}\end{array}\right),
\]
where $j-i=2l$. Then $r(X_{(i,j)})=X_{(i,j)}X_{(i,j)}^\ast =\vp^{(i,j)}$.  

We have the maximal compact subgroup $U(\Lam)\subset U(V_2)$. 
Our choice of Hermitian form implies that the group
\[
B=\left\{\left(\begin{array}{cc}t&\\&\overline{t}^{-1}\end{array}\right)\left(\begin{array}{cc}1&x\\&1\end{array}\right): t\in E^\times, x\in F\right\}
\]
is the $F$-points of a Borel subgroup of $U(V_2)$. The Iwasawa decomposition implies that
\begin{align*}
\Phi(\vp^{(i,j)})&=\int_{U(V_2)}\mathbf{1}_{\End(\Lam)}(X_{(i,j)}h)dh\\
		&=\sum_{b\in B/(B\cap U(\Lam))}\mathbf{1}_{\End(\Lam)}(X_{(i,j)}b).
\end{align*}
Now the product is of the form
\[
\left(\begin{array}{cc}1&\frac{\vp^i\zeta}{2}\\\vp^l&-\frac{\vp^{i+l}\zeta}{2}\end{array}\right)\left(\begin{array}{cc}1&u\\&1\end{array}\right)\left(\begin{array}{cc}t&\\&t^{-1}\end{array}\right)=\left(\begin{array}{cc}t&\frac{2u+\vp^i\zeta}{2t}\\t\vp^l&\vp^l\frac{2u-\vp^{i}\zeta}{2t}\end{array}\right).
\] 
Therefore, we need $\val(t)\geq0$, and 
\[
\val(2u+\vp^i\zeta)\geq \val(t), \text{ and }\val(2u-\vp^i\zeta)+l\geq \val(t).
\]
A set of representatives of the quotient $B/B\cap K$ is given by 
\[
\left(\begin{array}{cc}1&u\\&1\end{array}\right)\left(\begin{array}{cc}\vp^k&\\&\vp^{-k}\end{array}\right),
\]
with $k\in \zz_{\geq0}$ and $u\in F/\vp^{2k}\calo_F$. Since $u\in F$ and $b=\vp^i\zeta\in F\zeta$, we have $\val(u+b)=\val(u-b)=\min\{\val(u),\val(b)\}$. In particular, $i\geq k$, so that $i\geq 0$.

For each $0< k\leq i$, where $\val(t)=k$, we are free to pick any coset with $u\in \vp^k\calo/\vp^{2k}\calo$ so that $2k-1\geq \val(u)\geq k$ so that there are $q^k$ options for $u$, where $q=\#(\calo/\vp)$. Including the identity coset (with $k=0$), we obtain
\[
\Phi(\vp^{(i,n-i)})=\sum_{k=0}^iq^k.\qedhere
\]
\end{proof}


\subsection{Regular semi-simple elements and stable conjugacy} Lemma \ref{Lem: centralizers} tells us that the only relatively regular semi-simple elements $x\in \End(V_2)$ we need to consider are those such that $H_x\cong U(1)\times U(1)$, since rational and stable conjugacy agree for the other regular semi-simple conjugacy classes. Hereafter, we use the notation $\sim_{st}$ to denote stable conjugacy. 

\begin{Lem}\label{Lem: right stabilizer}
Assume $\de\in \Herm(V_2)$ has a stabilizer isomorphic to $U(1)\times U(1)$. There are values $a\in F$ and $\mu,\lam\in F^\times$ with  $\mu\lam\in \Nm(E^\times)\setminus{(F^\times)^2}$ such that $\de$ is rationally conjugate to
\[
\left(\begin{array}{cc}a&\lam\zeta\\\mu\zeta&a\end{array}\right).
\]
\end{Lem}

\begin{proof}

Note that if the centralizer of $\de$ in $U(V_2)$ is $U(1)\times U(1)$, then the centralizer of $\de$ in $\GL_2(E)$ is isomorphic to $E^\times\times E^\times$. Lemma \ref{Lem: centralizers} now implies that the eigenvalues of $\de$ lie in $F$, so there exist $a,b\in F$ such that
\[
\de\sim_{st}\left(\begin{array}{cc}a+b&\\&a-b\end{array}\right).
\]
Taking $a$ as above, $\mu=1$ and $\lam=(b/\zeta)^2$, and checking characteristic polynomials we see that $\de$ is stably conjugate 
\[
\ga=\left(\begin{array}{cc}a&\lam\zeta\\\zeta&a\end{array}\right).
\]
If $\de\sim \ga$, then we are done. Otherwise, consider
\[
\ga'=\left(\begin{array}{cc}a&\lam\vp^{-1}\zeta\\\vp\zeta&a\end{array}\right).
\]
The previous argument implies that $\de\sim_{st}\ga\sim_{st}\ga'$; we claim that $\ga$ and $\ga'$ are not rationally conjugate. This suffices to prove the lemma as there are only $2$ rational classes in the stable conjugacy class of $\de$ by Lemma \ref{Lem: centralizers}.

Let us consider the function
\[
\val\left(\begin{array}{cc}a&\lam\zeta\\\mu\zeta&a\end{array}\right)=\val(\mu).
\]
Note that $\val(\ga)\not\equiv\val(\ga')\pmod{2}.$ If $h= \left(\begin{array}{cc}x&y\\z&w\end{array}\right)\in U(V_2)$, 
\[
h\ga h^{-1}=\left(\begin{array}{cc}a+y\overline{w}-x\overline{z}\lam\zeta&x\overline{x}\lam\zeta-y\overline{y}\\w\overline{w}-z\overline{z}\lam\zeta&a-\overline{y}w+\overline{x}z\lam\zeta\end{array}\right).
\]
For $h\ga h^{-1}=\ga'$, we would need both
\[
y\overline{w}=x\overline{z}\lam\zeta\text{ and }\overline{y}w=\overline{x}z\lam\zeta.
\]
But this is impossible unless all terms are zero, in which case it is easy to check that
\[
\val(h\ga h^{-1})\equiv \val(\ga)\pmod{2}.
\]
It follows that $\ga$ and $\ga'$ are not rationally conjugate.
\end{proof}
\begin{Rem}\label{Rem: kostant section}
The representative
\[
\ga=\left(\begin{array}{cc}a&b^2\zeta^{-1}\\\zeta&a\end{array}\right)
\]
from the lemma lies in the rational points of a \emph{Kostant section} of the Chevalley invariant map
\[
\chi: \Herm(V_2)\lra \ft/W_T
\]
where $T\subset U(V_2)$ is the diagonal torus, $\ft=\Lie(T)$, and $W_T$ is the associated Weyl group. Our computation of the transfer factors below verifies that the transfer factor $\De$ agrees with Kottwitz's formulation in \cite{kottwitztransfer} in this case.
\end{Rem}

Thus for any regular element $\de\in\Herm(V_2)$, such that $T_\de\cong U(1)\times U(1),$ we may choose representatives of the two rational classes in the stable conjugacy class of $\de$ to be of the form
\[
\de_\pm=\left(\begin{array}{cc}a&\lam_\pm\zeta\\\mu_\pm\zeta&a\end{array}\right),
\]
where $\eta(\mu_\pm)=\pm1$.


\subsection{Orbital integrals}
  We begin with a simple lemma.
\begin{Lem}\label{Lem: simple orbital integrals}
Let $\de=\left(\begin{array}{cc}a&\lam\zeta\\\mu\zeta&a\end{array}\right)$ be as above, and denote by $X=X_{\mu,\lam}=\de-aI_2$ the off-diagonal matrix. Assume $\mu\lam\in\Nm(E^\times)\setminus{(F^\times)^2}$ and set $\val(\mu\lam)=2m$. Then
\[
\Orb(X_{\mu,\lam}, \bfun_{\End(\Lam)}) =\begin{cases}\qquad \displaystyle\sum_{k=0}^mq^k&: \quad \eta(\mu)=1\\ \qquad\displaystyle\sum_{k=0}^{m-1}q^k&: \quad \eta(\mu)=-1\end{cases}.
\]
\end{Lem}
\begin{proof}
We first consider the case that $\eta(\mu)=1$. Since $E/F$ is unramified, this restriction implies that $\val(\mu)=n$ is even. We may assume that $n=0$, since this does not change the conjugacy class of $\de$.

As above, the Iwasawa decomposition on $U(V_2)$ implies
\begin{align*}
\Orb(X_{\mu,\lam}, \bfun_{\End(\Lam)})&= \displaystyle \int_{U(V_2)} \bfun_{\End(\Lam)}(hX_{\mu,\lam}h^{-1})dh\\
								&=\sum_{h\in B\cap U(\Lam)\backslash B}\bfun_{\End(\Lam)}(h X_{\mu,\lam}h^{-1}).
\end{align*}
A set of representatives of the quotient $B/B\cap U(\Lam)$ is given by 
\[
h=\left(\begin{array}{cc}\vp^{-k}&\\&\vp^{k}\end{array}\right)\left(\begin{array}{cc}1&u\\&1\end{array}\right),
\]
with $k\in \zz$ and $u\in F/\vp^{2k}\calo_F$. Thus, we need to count $k$ and $u$ such that
\[
hX_{\mu,\lam}h^{-1} = \zeta\left(\begin{array}{cc}\mu u &\frac{\lam-\mu u^2}{\vp^{2k}}\\\mu \vp^{2k}&-\mu u\end{array}\right)
\]
is integral. This forces the inequalities 
\[
\val(u)\geq0,\quad \min\{m,\val(u)\}\geq k\geq0.
\]
We have used the fact that $\mu\lam\notin(F^\times)^2$ in identifying 
\[
\val(\lam-\mu u^2)=2 \min\{m,\val(u)\}.
\]From this the result follows easily in this case.

Now if we assume that $\eta(\mu)=-1$, then necessarily $\val(\mu)$ is odd, and we are free (up to conjugation) to assume $\val(\mu)=-1$. The result now follows from a similar argument as above.

\end{proof}

We now compute the orbital integrals $\Orb(\ga, \Phi_n)$.  Considering only regular elements with centralizer $T_\de\cong U(1)\times U(1)$, Lemma \ref{Lem: right stabilizer} implies we need only consider elements of the form
\[
\de_\pm=\left(\begin{array}{cc}a&\lam_\pm\zeta\\\mu_\pm\zeta&a\end{array}\right)\sim_{st}\left(\begin{array}{cc}a+b&\\&a-b\end{array}\right),
\]
where $\eta(\mu_\pm)=\pm1$. Then $\{\de_+,\de_-\}$ are representatives of the two conjugacy classes in the stable conjugacy class. 
The character $\ka$ constructed in Lemma \ref{Lem: endo character} gives our character 
\[
\kappa(\inv(\de_+,\de_\pm)) = \eta(\mu_\pm).
\]

\begin{Prop}\label{Prop: orbital computation}
Set $n_1=\val(b)$. Then 
\begin{equation}\label{eqn:kappa computation}
\Orb^\kappa(\de_+,\Phi_n)=\begin{cases}\quad q^{n_2}&: \text{  if  } \val(a+b),\val(a-b)\equiv 0\pod{2},\\\quad0&: \qquad\qquad\text{  otherwise.  }\end{cases}
\end{equation}
\end{Prop}

 \begin{proof} We explicitly compute the individual orbital integrals and then take the appropriate weighted sums. There are three cases to consider. For convenience, we record the results of these computations here. Set $\val(a)=n_1$, $\val(b)=n_2$, and $\val(\det(\de_\ast))=n\geq\min\{2n_1,2n_2\}$, where $\ast=\pm$.
\begin{enumerate}
\item\label{case1}
 If $n_1>n_2$, then $n_2=n/2\geq0$ and we have
\[
\Orb(\de_\ast, \Phi_n) = \sum_{k=0}^{n_2}\sum_{j=k}^{n_2}q^j-\begin{cases}\left[\frac{1+n_2}{2}\right]q^{n_2}: \text{  if  } \ast =+,\\\left[\frac{2+n_2}{2}\right]q^{n_2}: \text{  if  } \ast =-.\end{cases}
\]

\item\label{case2}
If $n_1<n_2$, then $n_1=n/2>0$ and we have
\[
\Orb(\de_\ast, \Phi_n) = \sum_{k=0}^{n_1}\sum_{j=k}^{n_2}q^j-\begin{cases}\left[\frac{1+n_1}{2}\right]q^{n_2}&: \text{  if  } \ast =+,\\\left[\frac{2+n_1}{2}\right]q^{n_2}&: \text{  if  } \ast =-.\end{cases}
\]

\item\label{case3}
If $n_1=n_2$, then $\val(\det(\de_\ast))=n\geq 2n_1$. If it is odd, then $\Orb(\de_\ast,\Phi)=0$. Otherwise,
\[
\Orb(\de_\ast, \Phi_n) = \sum_{k=0}^{n_2}\sum_{j=k}^{n_2}q^j-\begin{cases}\left[\frac{1+n_2}{2}\right]q^{n_2}&: \text{  if  } \ast =+,\\\left[\frac{2+n_2}{2}\right]q^{n_2}&: \text{  if  } \ast =-.\end{cases}
\]

\end{enumerate}

\begin{Rem}
The final case $n_1=n_2$ contains the \emph{nearly singular} case studied in \cite{polak2015exposing}. In that work, the author only considers elements $x$ with centralizer $E^\times\times E^\times$ in $U(4)$, which forces the eigenvalues of $r(x)$ to be norms. These have even valuation and we compute $n_2=2\val(x-y)=2V_m$, in his notation. 
\end{Rem}
In cases (\ref{case1}) and (\ref{case3}), we obtain
\[
\Orb^\kappa(\de_+,\Phi_n)=\Orb(\de_+, \Phi_n) -\Orb(\de_-, \Phi_n) =\begin{cases}\quad q^{n_2}&: \text{  if  } n_2\equiv 0\pod{2},\\\quad 0&: \text{  if  }  n_2\equiv 1\pod{2}.\end{cases}
\]
This is clearly equivalent to the statement of the proposition in case (\ref{case1}); for case (\ref{case3}), note that at most one of the eigenvalues $a\pm b$ has valuation greater than $n_2$. Without loss of generality, we may assume that $\val(a+b)=n_2$ and $\val(a-b) =k$. Then
\[
n_2+k=\val(a^2-b^2) = n,
\]
and $n$ being even implies $n_2\equiv k\pod{2}$. 
 
 Finally, case (\ref{case2}) becomes
\[
\Orb^\kappa(\de_+,\Phi_n)=\begin{cases}\quad q^{n_2}&: \text{  if  } n_1\equiv 0\pod{2},\\\quad 0&: \text{  if  }  n_1\equiv 1\pod{2},\end{cases}
\]
which immediately implies the claim.

We now prove these formulas. First assume that $n_1>n_2$. Then for any for $h\in U(V_2)$, for any $0\leq k\leq n/2$,
\[
h\de_\ast h^{-1}\in \vp^k\End(\Lam)_{\val(\det)=n-2k} \iff hX_{\mu,\lam}h^{-1}\in \vp^k\End(\Lam)_{\val(\det)=n-2k}.
\]
Indeed, since $\val(\det(\de_\ast))=2n_2=n$ and similarly for $X_{\mu,\lam}$, the only requirement is that the entries lie in $\vp^k\calo$. This holds for $aI_2$ by assumption, so that it holds for the entries of $h\de h^{-1}$ if and only if it holds for the entries of $hX_{\mu,\lam}h^{-1}$.

Using this and our computation of $\Phi_n$, we have
\begin{align*}
\Orb(\de_\ast, \Phi_n)&= \sum_{k=0}^{n/2}q^k\Orb\left(\de_\ast, \bfun_{\vp^k\End(\Lam)_{\val(\det)=n-2k}}\right)\\
	&= \sum_{k=0}^{n/2}q^k\Orb\left(X_{\mu,\lam}, \bfun_{\vp^k\End(\Lam)_{\val(\det)=n-2k}}\right)\\
&= \sum_{k=0}^{n/2}q^k\Orb\left(\vp^{-k}X_{\mu,\lam},\bfun_{\End(\Lam)_{\val(\det)=n-2k}}\right)\\
&= \sum_{k=0}^{n/2}q^k\Orb\left(\vp^{-k}X_{\mu,\lam},\bfun_{\End(\Lam)}\right).
\end{align*}
This last reduction follows since $\vp^{-k}X_{\mu,\lam}$ has the correct determinant, so that the orbital integrals over the test functions $\bfun_{\End(\Lam)_{\val(\det)=n-2k}}$ and $\bfun_{\End(\Lam)}$ agree.
Thus, we are reduced to computing the orbital integral
\[
\Orb\left(\vp^{-k}X_{\mu,\lam},\bfun_{\End(\Lam)}\right).
\]
Now let $\de_\ast=\de_+$. By Lemma \ref{Lem: simple orbital integrals},
\[
\Orb\left(\vp^{-k}X_{\mu,\lam},\bfun_{\End(\Lam)}\right)=\begin{cases}\displaystyle\sum_{j=0}^{(n-2k)/2}q^j :\quad k\text{ even,}\\\displaystyle\sum_{j=0}^{(n-2k)/2-1}q^j :\quad k\text{ odd,}\end{cases}
\]
so that 
\[
\Orb(\de_+, \Phi_n)=\sum_{k=0}^{n/2}\sum_{j=k}^{n/2}q^j-\left[\frac{1+(n/2)}{2}\right]q^{n/2}.
\]
The computation is similar for $\de_\ast=\de_-$, and we find
\[
\Orb(\de_-, \Phi_n)=\sum_{k=0}^{n/2}\sum_{j=k}^{n/2}q^j-\left[\frac{2+(n/2)}{2}\right]q^{n/2}.
\]

In the case that $n_1<n_2$, there is a similar reduction. Indeed, since the valuation is correct, $h\de_\ast h^{-1}\in \vp^k\End(\Lam)_{\val(\det)=n-2k}$ if and only if $h X_{\mu,\lam} h^{-1}\in \vp^k\End(\Lam)$. 
Writing
\[
h\vp^{-k}\de_\ast h^{-1}= \vp^{-k}aI_2+ h \vp^{-k}X_{\mu,\lam} h^{-1}.
\]
Since $k\leq n_1<n_2$, it is clear that integrality of the left-hand side is equivalent to the integrality of $h\vp^{-k}X_{\mu,\lam} h^{-1}$. Therefore, we consider the orbital integral
\begin{align*}
\Orb(\de_\ast, \Phi_n)&= \sum_{k=0}^{n_1}q^k\Orb\left(\vp^{-k}X_{\mu,\lam}, \bfun_{\End(\Lam)}\right),
\end{align*}
which is computed as above.

A similar argument  works in the case that $n_1=n_2$, provided $k\leq n_1$. In general,
\begin{align*}
\Orb(\de_\ast, \Phi_n)&= \sum_{k=0}^{n_1}q^k\Orb\left(\vp^{-k}X_{\mu,\lam}, \bfun_{\End(\Lam)}\right)\\&+\sum_{k=n_1+1}^{n/2}q^k\Orb\left(\de_\ast, \bfun_{\vp^k\End(\Lam)_{\val(\det)=n-2k}}\right).
\end{align*}
The first set of integrals are computed as above. Consider now the case that $n_1<k\leq n/2$. Considering the sum
\[
h\vp^{-k}\de_\ast h^{-1} = \vp^{-k}aI_2+ h \vp^{-k}X_{\mu,\lam}h^{-1},
\]
 since $\vp^{-k}a\notin \calo_E$, we find that $h\vp^{-k}\de_\ast h^{-1}\in \End(\Lam)$ if and only if $h \vp^{-k}X_{\mu,\lam} h^{-1}\in \End(\Lam)-\vp^{-k}aI_2$. In particular, the lack of integrality of $h\vp^{-k}X_{\mu,\lam}h^{-1}$ is precisely canceled by the central term.

 We show this such a cancellation is not possible. Indeed, writing $h=kb$ for $k\in \GL(\Lam)$ and $b\in B$, then $h\vp^{-k}\de_\ast h^{-1}\in \End(\Lam)$ if and only if $b\vp^{-k}\de_\ast b^{-1}\in \End(\Lam)$, so that we may reduce to elements in the Borel subgroup as before. As previously noted, we may assume that our representatives are of the form
\[
h=\left(\begin{array}{cc}\vp^{-m}&\\&\vp^{m}\end{array}\right)\left(\begin{array}{cc}1&u\\&1\end{array}\right),
\]
with $m\in \zz$ and $u\in F$. Thus, we have
\[
bX_{\mu,\lam}b^{-1} = \left(\begin{array}{cc}\mu u\zeta &\frac{\lam\zeta-\mu u^2\zeta}{\vp^{2m}}\\\mu \vp^{2m}\zeta&-\mu u\zeta\end{array}\right).
\]
But $u\mu\in F$ so that it is not possible for $a+u\mu\zeta\in \vp^{k}\calo_E$ when $a\notin\vp^{k}\calo$. It follows that the orbital integrals $\Orb\left(\de_\ast, \bfun_{\vp^k\End(\Lam)_{\val(\det)=n-2k}}\right)$ vanish for $k>n_1=\val(a)$.
\end{proof}

\subsection{The endoscopic side}
Let $\de_+\in \Herm(V_2)^{rss}$ be as in the previous section. Up to stable conjugacy,
\[
\de_+\sim_{st}\left(\begin{array}{cc}a+b&\\&a-b\end{array}\right),
\]
and we send $\de_+\to (a+b,a-b)\in \Herm(V_1)\oplus \Herm(V_1)\cong F^2$. Recall that the split Hermitian form on $V_1= Ev$ is given so that $\la v,v\ra\in \Nm(E)$.

The relative orbital integrals for this action are trivial: in the case of a single copy of $(U(V_1)\times U(V_1),\fu(V_1\oplus V_1)_1)$, the contraction map $r:\End(V_1)\to \Herm(V_1)$ corresponds to the field norm $\Nm_{E/F}:E\to F$. Moreover, the action of $U(V_1)\times U(V_1)$ on $E$ is given by
\[
(g,g')\cdot e = ge\overline{g'},
\]
so that the contraction map is invariant with respect to both copies of $U(V_1)$ and takes the $E^\times\times U(V_1)$ action to 
\[
(g,h)\cdot e\overline{e} = ge\overline{e}\overline{g} = \Nm(g)\Nm(e).
\]
For any smooth integrable function $\phi$ on $\fu(V_1\oplus V_1)_1\cong E$ that is $\calo^\times_E$-invariant and any $x\in E^\times$,
\begin{align*}
\RO(x,\phi):&=\int_{U(V_1)\times U(V_1)}\phi(gx\overline{g'})dgdg'\\& = \int_{U(W_1)}r_\ast\phi(\Nm(g)\Nm(x))dg=r_\ast\phi(\Nm(x))=\phi(x).
\end{align*}
We introduce the function $\Phi^{\ka}:\Herm(V_1)\times\Herm(V_1)\to \cc$
\begin{equation*}
\Phi^\kappa(x,y) = \begin{cases} \qquad 1 \qquad: \val(x)\equiv\val(y)\equiv0\pmod{2},\\ \quad\quad 0\qquad: \text{otherwise}.\end{cases}
\end{equation*}
Letting $\bfun_{\End(\Lam_1)}\otimes \bfun_{\End(\Lam_1)}$ denote the basic function for the endoscopic symmetric space, it is easy to check that
\[
\Phi^\ka=r_{(\al_0,\be_0),!}\left(\bfun_{\End(\Lam_1)}\otimes \bfun_{\End(\Lam_1)}\right).
\]
\begin{Prop}\label{Prop: endoscopic side} 
 For $\de_+$ as in Proposition \ref{Prop: orbital computation}, we have
\begin{equation*}
\SO((a+b,a-b),\Phi^{\kappa})=\Delta((a+b,a-b), \de_+)\Orb^\ka(\de_+,\Phi).
\end{equation*}
\end{Prop}

\begin{proof} Our previous remarks allow us to compute the left-hand side:
\[
\SO((x,y),\Phi^{\kappa})=\begin{cases} \qquad 1 \qquad: \val(x),\val(y)\equiv0\pmod{2},\\ \quad\quad 0\qquad: \text{otherwise}.\end{cases}
\]

For the right-hand side, some care must be taken with the transfer factor. When the matching $\de_+\mapsto (a+b,a-b)$ is a nice matching in the sense of Section \ref{Section: unitary endoscopy}, the transfer factor (\ref{nice matching}) may be computed as 
\[
\De((a+b,a-b),\de_+)= (-q)^{-\val(b)}=(-1)^{n_2}q^{-n_2},
\]
using the notation from Proposition \ref{Prop: orbital computation}. This matching is nice if and only if the restriction of the Hermitian form of $V_2$ to each of the two eigenlines $V_2=L_1\oplus L_2$ of $\de_+$ corresponds to a split Hermitian form. A simple computation shows that this is the case if and only if $n_2=\val(b)$ is even. 
When $n_2$ is odd, then $\de_-$ is a nice match with $(a+b,a-b)$ so that
\[
\De((a+b,a-b),\de_+)=-\De((a+b,a-b),\de_-)= -(-1)^{n_2}q^{-n_2}=q^{-n_2}.
\]
By Proposition \ref{Prop: orbital computation}, we see that $\Orb^\ka(\de_+,\Phi)$ vanishes unless both eigenvalues $a+b$ and $a-b$ are norms. Comparing with (\ref{eqn:kappa computation}), we obtain the desired identity.

\end{proof}

\newpage
\begin{appendix}
\section{Comment on Torsors} In this appendix, we record some elementary properties of functions on torsors. The proofs are standard exercises which we omit.

Let $\G$ be an affine algebraic group over a field $F$. There is a well-known correspondence
\[
\{\text{$F$-torsors of $\G$}\}/\sim\:\longleftrightarrow H^1(F,\G):=H^1_{cont}(\Gal(F^{sep}/F),\G(F^{sep})),
\]
where $\G(F^{sep})$ is endowed with the discrete topology. Let $f:\X\to \Y$ be a $\G$-torsor and let $[\al]\in H^1(F,G)$. As explained in \cite[Chapt. 1, \S 5]{SerreGalois}, we may twist $\X$ by $\al$ to obtain another $\G$-torsor $f_\al: \X_\al\to \Y$. Setting $X_\al=\X_\al(F)$ and $Y=\Y(F)$, the following result is standard.
\begin{Prop}\label{Thm: torsor rational}
Let $f:\X\to \Y$ be a $\G$-torsor. Then
\[
Y=\bigsqcup_{[\al]\in H^1(F,\G)}f_\al(X_\al).
\]
This is true even if $X_\al=\emptyset$ for some $\al\in H^1(F,\G)$.
\end{Prop}

Suppose that $\X=\mathbf{V}$ is a linear representation of $\G$. In this case, Hilbert's theorem 90 implies that we for any $[\al]\in H^1(F,G)$, 
\[
\mathbf{V}_\al\cong \mathbf{V}
\]
 and that these isomorphisms may be chosen compatibly so that each pure inner form $\G_\al$ acts on the variety $\mathbf{V}$. This is discussed in greater detail in \cite{bhargava2015arithmetic}, for example.
 \begin{Cor}\label{Cor: torsor rational}Suppose that $\mathbf{V}$ is a linear representation of $\G$. Let $f: \X\to \Y$ is a $\G$-torsor and assume that there exists  $\G$-equivariant embeddings $\X\subset \mathbf{V}$ such that the image is $\GL(\mathbf{V})$-stable and $\Y\subset \mathbf{V}//\G$ such that the diagram
 \[
 \begin{tikzcd}
 \X\ar[d,"f"]\ar[r]&\mathbf{V}\ar[d,"q"]\\
 \Y\ar[r]&\mathbf{V}//\G,
 \end{tikzcd}
 \] where $q$ is the canonical quotient map, commutes. Then for any $[\al]\in H^1(F,G)$, $\X_\al\cong \X$ and there exist twisted torsor maps $f_\al: \mathbf{X}\to \Y$ such that
 \[
 Y=\bigsqcup_{[\al]\in H^1(F,\G)}f_\al(X).
 \]
\end{Cor}

Now suppose that $F$ is a local field. Let $\G$ be a reductive group over $F$ and $f:{\X}\to {\Y}$ is a ${\G}$-torsor of $F$-varieties such that there exists embeddings producing a diagram of $\G$-varieties as in the corollary. We introduce the notation $f/G:[Y/G]\to Y$ to denote the map of Hausdorff spaces
\[
f/G:[Y/G]:=\bigsqcup_{[\al]\in H^1(F,\G)}X\xrightarrow{f_\al}Y.
\]
\begin{Prop}\label{Prop: surjective schwartz}
The map $f/G$ is a submersion. In particular, it induces a surjective map
\[
\left(f/G\right)_!:\bigoplus_{[\al]\in H^1(F,\G)}C_c^\infty (X)\to C_c^\infty(Y)
\]
where for a function $\phi=\sum_\al \phi_\al$ and for any $y\in Y$ 
\[
\left(f/G\right)_!\phi(y) = \int_{G_\al}\phi_\al(g\cdot x)dg_\al,
\]
where $x\in (f/G)^{-1}(y)$, and $dg_\al$ is the Haar measure on $G_\al$ chosen in such a way that all measures $dg_\be$ are compatible via the inner twisting from $\G$.
\end{Prop}
\end{appendix}


\bibliographystyle{alpha}

\bibliography{bibs_endoscopy}
\end{document}